\documentclass[a4paper,12pt]{article}
\usepackage{amsmath, amssymb, amsthm}

\newtheorem{thm}{Theorem}[section]
\newtheorem{cor}[thm]{Corollary}
\newtheorem{lem}[thm]{Lemma}
\newtheorem{prop}[thm]{Proposition}
\newtheorem{defi}[thm]{Definition}
\newtheorem{que}[thm]{Question}
\newtheorem{rem}[thm]{Remark}

\newcommand{\la}{\langle}
\newcommand{\ra}{\rangle}

\newcommand{\mr}{\mathbf{R}}
\newcommand{\mn}{\mathbf{N}}

\newcommand{\ord}{\text{ord}}

\title{Factorial and Noetherian Subrings of Power Series Rings}
\author{Damek Davis and Daqing Wan \\ \small Department of Mathematics \\\small University of California
\\\small Irvine, CA 92697-3875 \\\small davisds@uci.edu \\\small dwan@math.uci.edu}

\begin{document}
  \maketitle

\begin{abstract}
Let $F$ be a field.  We show that certain subrings contained between the polynomial ring $F[X] = F[X_1, \cdots, X_n]$ and the
power series ring $F[X][[Y]] = F[X_1, \cdots, X_n][[ Y]]$ have Weierstrass Factorization, which allows us to deduce both unique factorization and the Noetherian property.  These intermediate subrings are obtained from elements of $F[X][[ Y]]$ by bounding their total $X$-degree above by a positive real-valued monotonic up function $\lambda$ on their $Y$-degree. These rings arise naturally in studying $p$-adic analytic variation of zeta functions over finite fields. Future research into this area may study more complicated subrings in which $Y = (Y_1, \cdots, Y_m)$ has more than one variable, and for which there are multiple degree functions, $\lambda_1, \cdots, \lambda_m$. Another direction of study would be to generalize these results to $k$-affinoid algebras.
\end{abstract}

\section{Introduction}
Let $R$ be a commutative ring with unity, and let $S_k$ be the set of polynomials in $R[X, Y] = R[X_1, \cdots, X_n][ Y_1, \cdots, Y_m]$ that are homogeneous in $Y$ of degree $k$.  Every element of $R[X][[Y]]$ can be written uniquely in the form
\begin{eqnarray}
 f &=& \sum_{k=0}^\infty f_{k}(X, Y),
\end{eqnarray}
where $f_k(X, Y)$ is an element of $S_k$. In this expansion, there is no restriction on $\deg_X f_k(X, Y)$. Motivated by several applications to the $p$-adic theory of zeta functions over finite fields, we want to consider subrings of $R[X][ Y]]$ in which $\deg_X (f_k)$ is bounded above by some function $\lambda$. In particular, let $\lambda : \mr_{\geq 0} \rightarrow \mr_{\geq 0}$ be a monotonic up function. We call $\lambda$ a growth function. Following Wan \cite{Wan1}, we define a subring of $R[X][[Y]]$ as follows:
\begin{eqnarray}
 R[X; Y, \lambda] &=& \{f = \sum_{k=0}^\infty f_{k}(X, Y) :
 f_k  \in S_k, \deg_X(f_k) \leq C_f\lambda(k), \text{ for } k \gg 0 \},
\end{eqnarray}
where $C_f$ is a constant depending only on $f$. Since $\lambda$ is monotonic up, it satisfies the trivial inequality,
\begin{eqnarray*}
 \lambda(x) + \lambda(y) &\leq& 2\lambda(x+y)
\end{eqnarray*}
for all $x$ and $y$ in $\mr_{\geq 0}$. From this inequality, it is clear that $R[X; Y, \lambda]$ is an $R[X]$-algebra, which contains $R[X]$.

If $\lambda$ is invertible, we have the following equivalent definition:
\begin{eqnarray}
 R[X; Y, \lambda] &=& \{g = \sum_{d=0}^\infty g_{d}(X, Y) : g_d  \in A_d, \ord_Y(g_d) \geq \lambda^{-1}(C_gd), \text{ for } d \gg 0\},
\end{eqnarray}
where $A_d$ is the subset of elements of $R[[Y]][X]$, which are homogeneous of degree $d$ in $X$, and $\ord_Y (g_d)$ is the largest integer $k$ for which $g_d$ is an element of $Y^kR[X][[Y]]$.

It is clear that for any positive constant $c >  0$, $R[X; Y, c\lambda] = R[X; Y, \lambda]$. If $\lambda(x)$ is a positive constant, then $R[X; Y, \lambda] = R[[Y]][X]$. If $\lambda(x) = x$ for all $x$ in $\mr_{\geq 0}$, then $R[X; Y, \lambda]$ is called the over-convergent subring of $R[X][[Y]]$, which is the starting point of Dwork's $p$-adic theory for zeta functions. In both of these cases, if $R$ is noetherian, it is known that $R[X; Y, \lambda]$ is noetherian: when $\lambda$ is constant, the result follows from Hilbert's Basis Theorem; the case in which $\lambda(x) = x$ is proved in Fulton \cite{Fulton}. More generally, if $R$ is noetherian and $\lambda$ satisfies the following inequality,
\begin{eqnarray*}
 \lambda(x) + \lambda(y) \leq \lambda(x+y) \leq \lambda(x)\mu(y)
\end{eqnarray*}
for all sufficiently large $x$ and $y$, where $\mu$ is another positive valued function such that $\mu(x) \geq 1$ for all $x$ in $\mr_{\geq0}$, then $R[X; Y, \lambda]$ is also noetherian as shown in Wan \cite{Wan1}. For example, any exponential function $\lambda(x)$ satisfies the above inequalities. In this case the ring is particularly interesting because it arises naturally from the study of unit root F-crystals from geometry, see Dwork-Sperber \cite{DS} and Wan \cite{Wan2} for further discussions.

The first condition, $\lambda(x) + \lambda(y) \leq \lambda(x+y)$, is a natural assumption because it ensures that elements of the form $(1-XY)$ are invertible, a vital condition to this paper. If $\lambda$ does not grow at least as fast as linear, then $(1-XY)^{-1} = 1 + \sum_{i=1}^{\infty} X^kY^k$ is not an element of $R[X; Y, \lambda]$.  It is not clear, however, if the second condition, $\lambda(x+y) \leq \lambda(x)\mu(y)$, can be dropped. In fact, we have the following open question from Wan \cite{Wan1}.
\begin{que} Let $R$ be a noetherian ring. Let $\lambda(x)$ be a growth function satisfying $\lambda(x) + \lambda(y) \leq \lambda(x+y)$. Is the intermediate ring $R[X; Y, \lambda]$ always noetherian?
\end{que}

This question is solved affirmatively in this paper if $R$ is a field and there is only one $Y$ variable.

Throughout this paper we assume that $R=F$ is a field, and that $\lambda$ grows at least as fast as linear, i.e. $\lambda(x) + \lambda(y) \leq \lambda(x+y)$ for all $x, y \geq 0$. Further, we assume that $\lambda(0) = 0$ and $ \lambda(\infty) = \infty$, because normalizing $\lambda$ this way does not change $F[X; Y, \lambda]$. Without loss of generality we also assume that $\lambda$ is strictly increasing. Finally, we assume that $F[X; Y, \lambda]$ has only one $Y$ variable.  We call an element
\begin{eqnarray*}
 g = \sum_{d=0}^\infty g_d(X_1, \cdots, X_{n-1}, Y)X_n^d
\end{eqnarray*}
in $F[X; Y, \lambda]$ $X_n$-distinguished of degree $s$ if $g_s$ is a unit in $F[X_1,\cdots, X_{n-1};Y,\lambda]$, and $\ord_Y(g_d) \geq 1$ for all $d> s$. The main result of this paper is the following

\begin{thm}
Under the above assumptions, we have
\begin{enumerate}
 \item (Euclidean Algorithm) \textit{Suppose that $g$ is $X_n$-distinguished of degree $s$ in $F[X; Y, \lambda]$, and that $f$ is an element of $F[X; Y, \lambda]$.  Then there exist unique elements, $q$ in $F[X; Y, \lambda]$, and $r$ in the polynomial ring $F[X_1, \cdots, X_{n-1}; Y, \lambda][X_n]$ with ${\rm deg}_{X_n}(r) <s$, such that $f = qg + r$.}
\item (Weierstrass Factorization) \textit{Let $g$ be $X_n$-distinguished of degree $s$. Then there exists a unique monic polynomial $\omega$ in $F[X_1, \cdots, X_{n-1}; Y, \lambda][X_n]$ of degree $s$ in $X_n$ and a unique unit $e$ in $F[X ; Y, \lambda]$ such that $g = e\cdot \omega$.  Further, $\omega$ is distinguished of degree $s$.} 
\item (Automorphism Theorem) \textit{Let $g(X,Y) = \sum_{\mu}g_\mu(Y)X^\mu$ be an element of $F[X; Y,\lambda]$ where $\mu = (\mu_1, \cdots, \mu_n)$ and $X^\mu = X_1^{\mu_1}\cdots X_n^{\mu_n}$. If $g_\mu(Y)$ is not divisible by $Y$ for some $\mu$,  where $\mu_n > 0$, then there exists an automorphism $\sigma$ of $F[X; Y, \lambda]$ such that $\sigma(g)$ is $X_n$-distinguished.}
 \item \textit{$F[X; Y, \lambda]$ is noetherian and factorial.}
\end{enumerate}
\end{thm}

The Euclidean algorithm is the key part of this theorem. Our proof of this algorithm follows Manin's proof of the analogous result for power series rings as written in Lang \cite{Lang}, except that we have to keep careful track of more delicate estimates that arise from the general growth function $\lambda(x)$. The other results are classical consequences of this algorithm, which are proved in this paper, but the techniques are essentially unchanged from techniques utilized in proofs of analogous results for power series rings as given in Bosch, etc. \cite{Bosch}.

This topic is also motivated by a considerable body of work concerning ``$k$-affinoid'' algebras from non-Archimedean
analysis.  Let $k$ be a complete non-Archimedean valued field, with a non-trivial valuation, and define $T_n = k\la X_1, \cdots, X_n \ra$, Tate's algebra, to be the algebra of strictly convergent power series over $k$: $T_n = \{ \sum_{\mu} a_\mu X^\mu : |a_\mu| \stackrel{|\mu| \rightarrow \infty}{\rightarrow}  0 \}$. The algebra, $T_n$, is a noetherian and factorial ring with many useful properties, and it is the basis for studying $k$-affinoid algebras, see Bosch etc \cite{Bosch}. A $k$-algebra, $A$, is called $k$-affinoid if there exists a continuous epimorphism, $T_n \rightarrow A$, for some $n \geq 0$. Given $\rho = (\rho_1, \cdots, \rho_n)$ in $\mr^n$, where $\rho_i > 0$ for each $i$, one can define
\begin{eqnarray*}
T_{n}(\rho) &=& \{ \sum_{\mu} a_\mu X^{\mu} \in k[[X_1, \cdots, X_n]] : |a_\mu|\rho_1^{\mu_1}\cdots \rho_n^{\mu_n} \stackrel{|\mu| \rightarrow \infty} {\rightarrow} 0\}.
\end{eqnarray*}
Note that $T_n(1,\cdots, 1) = T_n$. Furthermore, $T_n(\rho)$ is $k$-affinoid if, and only if, $\rho_i$ is an element of $|k_a^\ast|$ for all $i$, where $k_a$ is the algebraic closure of $k$, from which one can immediately verify that it is noetherian. It is shown by van der Put in \cite{Put} that this ring is noetherian for any $\rho$ in $\mr^n$, where $\rho_i > 0$ for each $i$. Define the Washnitzer algebra $W_n$ to be
\begin{eqnarray*}
 W_n &=& \bigcup_{\rho\in \mr^n, \rho_i > 1} T_n(\rho).
\end{eqnarray*}
It is shown in G\"{u}ntzer \cite{G} that $W_n$ is noetherian and factorial. A motivating study of $W_n$ is given by Grosse-K\"{o}nne \cite{Gr}. This overconvergent ring $W_n$ is also the basis (or starting point) of the Monsky-Washnitzer formal cohomology and the rigid cohomology.

More generally, for a growth function $\lambda(x)$, we can also define
\begin{eqnarray*}
 T_n(\rho, \lambda) &=& \{ \sum_{\mu} a_\mu X^\mu \in k[[X_1, \cdots, X_n]] : |a_\mu|\rho_1^{\lambda^{-1}(\mu_1)}\cdots \rho_n^{\lambda^{-1}(\mu_n)} \stackrel{|\mu| \rightarrow \infty}{\rightarrow} 0 \}.
\end{eqnarray*}
Similarly, define
\begin{eqnarray*}
 W_n(\lambda) &=& \bigcup_{\rho\in \mr^n, \rho_i > 1} T_n(\rho,
 \lambda).
\end{eqnarray*}
If $\lambda$ is invertible, $W_n(\lambda)$ most closely resembles the ring $k[X; Y, \lambda] = k[X_1, \cdots, X_n; Y, \lambda]$ studied in this paper. If $\lambda(x) = cx$ for some $c > 0$ and all $x$ in $\mr_{\geq 0}$, then $W_n(\lambda) = W_n$ is the Washnitzer algebra. Similarly, $T_n((1, \cdots, 1), \lambda) = T_n$ for all $\lambda$, and $T_n(\rho, \text{id}) = T_n(\rho)$.

The results of this paper suggest that there may be a $p$-adic cohomology theory for more general growth functions $\lambda(x)$ (other than linear functions), which would help to explain the principal zeroes of Dwork's unit root zeta function \cite{Dw}, \cite{Wan2} in the case when $\lambda$ is the exponential function. This is one of the main motivations for the present paper.

{\bf Acknowledgments}. We would like to thank Christopher Davis for informing us of several relevant references.

\section{Results and Proofs}
For the rest of the paper, we assume that $F$ is a field and that $p$ is a fixed positive real number greater than one.
\begin{defi}
 Define $| \; |_\lambda$ on $F[[Y]]$ by $| f(Y) |_\lambda = \frac{1}{p^{\lambda(\ord_Y(f))}}$ for all $f$ in $F[[Y]]$.
\end{defi}

This is basically the $Y$-adic absolute value on $F[[Y]]$, re-scaled by the growth function $\lambda(x)$.

\begin{prop}
 $(F[[Y]], | \; |_\lambda$) is a complete normed ring.
\end{prop}
\begin{proof}
 We defined $\lambda(0) = 0$ and $\lambda(\infty) = \infty$, so $|a|_\lambda = 0$ if, and only if, $a = 0$ (because $\lambda$ is strictly increasing), and $|c_0| = 1$ for all $c_0$ in $F^\times$. Suppose that $f$ and $g$ are elements of $F[[Y]]$, then $\ord_Y(fg) = \ord_Y(f) + \ord_Y(g)$, and $\lambda(\ord_Y(f) + \ord_Y(g)) \geq \lambda(\ord_Y(f)) + \lambda(\ord_Y(g))$. Thus,
\begin{eqnarray*}
 |fg|_\lambda &=& \frac{1}{p^{\lambda(\ord_Y(f) + \ord_Y(g))}} \\
&\leq& \frac{1}{p^{\lambda(\ord_Y(f))+\lambda(\ord_Y(g))}} \\
&=& |f|_\lambda |g|_\lambda.
\end{eqnarray*}
Similarly, since $\ord_Y(f+g) \geq \min\{\ord_Y(f), \ord_Y(g)\}$
\begin{eqnarray*}
 |f+g|_\lambda &=& \frac{1}{p^{\lambda(\ord_Y(f+g))}} \\
&\leq& \frac{1}{p^{\lambda(\min\{\ord_Y(f), \ord_Y(g)\})}} \\
&=& \max\{|f|_\lambda, |g|_\lambda\}.
\end{eqnarray*}

 To show completeness, if $\left(f^{(i)}\right)_{i=1}^\infty$ is a Cauchy sequence with respect to the standard $Y$-adic norm $|\;|$ on $F[[Y]]$, then $f^{(i)}$ converges to an element $f$ in $F[[Y]]$.  Thus, $|f - f^{(i)}| = \frac{1}{p^{\ord_Y(f-f^{(i)})}}$ converges to $0$ as $i$ approaches $\infty$, and so $\ord_Y(f - f^{(i)})$ must approach $\infty$.  This can happen if, and only if, the corresponding sequence, $\lambda(\ord_Y(f - f^{(i)}))$, approaches $\infty$, as desired.
\end{proof}

\begin{rem} A norm which only satisfies $|ab| \leq |a||b|$, instead of strict equality is sometimes called a pseudo-norm; we disregard the distinction in this paper.
\end{rem}

We can write any element $f$ in $F[X; Y, \lambda]$ in the following form:
\begin{eqnarray*}
 f(X,Y) &=& \sum_{\mu} f_\mu(Y)X^\mu
\end{eqnarray*}
where $\mu = (\mu_1, \cdots, \mu_n)$ is a tuple of positive integers, and $X^\mu = X_1^{\mu_1}\cdots X_n^{\mu_n}$. This form and the above norms allow us to formulate two equivalent definitions for $F[X; Y, \lambda]$:
\begin{eqnarray*}
 F[X; Y, \lambda] &=& \{f = \sum_{\mu}^\infty f_{\mu}( Y)X^\mu : f_\mu  \in F[[Y]], |f_\mu| p^{\lambda^{-1}(C_f|\mu|)} \stackrel{|\mu| \rightarrow \infty}{\rightarrow} 0\} \\
&=& \{f = \sum_{\mu}^\infty f_{\mu}( Y)X^\mu : f_\mu  \in F[[Y]], |f_\mu|_\lambda p^{C_f|\mu|} \stackrel{|\mu| \rightarrow \infty}{\rightarrow} 0\}
\end{eqnarray*}
where $|\mu| = \mu_1 + \cdots + \mu_n$.
\begin{defi}
 For all $c$ in $\mr^{n}, c_i > 0$, define
\begin{eqnarray*}
 F[X; Y, \lambda]_c &=& \{\sum_{\nu} f_\mu X^\mu : \left| f_\mu \right|_\lambda p^{c\cdot \mu} \stackrel{|\mu| \rightarrow \infty}{\rightarrow} 0\}
\end{eqnarray*}
where $c \cdot \mu = c_1\mu_1 + \cdots + c_n\mu_n$.
\end{defi}
\begin{defi}
Define $\| \; \|_{\lambda, c}$ on $F[X; Y, \lambda]_c$ by
\begin{eqnarray*}
\|f\|_{\lambda, c} &=&  \max |f_\mu|_\lambda p^{c\cdot \mu}.
\end{eqnarray*}
\end{defi}
It's easy to see that $F[[Y]] \subset F[X; Y, \lambda]_c \subseteq F[X; Y, \lambda]_{c'}$ if $c_i' \leq c_i$ for $i = 1, \cdots, n$.
\begin{prop}
The function $\| \; \|_{\lambda, c}$ is a non-Archimedean norm on $F[X; Y, \lambda]_c$.
\end{prop}
\begin{proof}
 On $F[[Y]]$, $\|\;\|_{\lambda, c}$ reduces to $| \; |_\lambda$.  Suppose that $f$ and $g$ are elements of $F[X; Y,\lambda]_c$. Then
\begin{eqnarray*}
 \|f + g\|_{\lambda, c} &=& \max_{\mu}\{ |f_\mu+g_\mu|_\lambda p^{c\cdot \mu}\}\\
&\leq& \max_{\mu}\{\max\{|f_\mu|_\lambda, |g_\mu|_\lambda\}p^{c\cdot \mu}\} \\
&\leq& \max\{ \|f\|_{\lambda, c}, \|g\|_{\lambda, c}\}.
\end{eqnarray*}
Next,
\begin{eqnarray*}
 \|fg\|_{\lambda, c} &=& \max_{\sigma}\{ \left|\left(\sum_{\mu + \nu = \sigma} f_\mu g_\nu\right) \right|_\lambda p^{c\cdot \sigma} \} \\
&\leq& \max_{\sigma}\{ \max_{\mu + \nu = \sigma}\{ |f_\mu g_\nu|_\lambda\} p^{c\cdot \sigma} \} \\
&\leq& \max_{\mu, \nu}\{ |f_\mu|_\lambda |g_\nu|_\lambda p^{c\cdot ( \mu + \nu)} \} \\
&=& \|f\|_{\lambda, c}\|g\|_{\lambda, c}.
\end{eqnarray*}
\end{proof}

\begin{prop}
 $\displaystyle F[X; Y, \lambda] = \bigcup_{c \in \mr^n, c_i> 0} F[X; Y, \lambda]_c$
\end{prop}
\begin{proof}
Suppose that $f$ is an element of $F[X; Y, \lambda]$, then $|f_\mu|_\lambda p^{(C_f, \cdots, C_f)\cdot \mu}$ converges to $0$ as $|\mu|$ approaches $\infty$.  Conversely, if $f$ is an element of $F[X; Y, \lambda]_c$, let $C_f = \min_{i} c_i$.
\end{proof}

\begin{lem}
 Suppose $f$ is an element of $F[X; Y,\lambda]$. Then $f(X,Y)$ is invertible if, and only if, $f \equiv c_0 \mod (Y)$ where $c_0$ is a unit in $F$. If an element $f$ in $F[X; Y, \lambda]_c$ is invertible in $F[X; Y, \lambda]$, then $f^{-1}$ is an element of $F[X; Y, \lambda]_c$. Further, if $\|f\|_{\lambda, c} \leq 1$, then $\| f^{-1}\|_{\lambda, c} \leq 1$.
\end{lem}
\begin{proof}
If $f(X,Y)$ is invertible, then it is an invertible polynomial modulo $(Y)$. Therefore, $f$ is a non-zero unit modulo $(Y)$.

If $f \equiv c_0 \mod (Y)$ for $c_0$ in $F^\times$, we can write $f = c_0(1 - g(X,Y))$ as an element of $F[X;Y,\lambda]_c$ for some $c>0$. Then
\begin{eqnarray*}
 f^{-1} &=&  c_0^{-1}\left(1 + \sum_{k=1}^{\infty}g(X,Y)^k\right) \\
&=& c_0^{-1}\left(1 + \sum_{k=1}^\infty\sum_{j=1}^\infty
\sum_{\substack{\mu^{(1)} + \cdots + \mu^{(k)} = \sigma \\
|\sigma| = j} }\prod_{i=1}^k g_{\mu^{(i)}}(Y) X^\sigma\right)
\end{eqnarray*}
 Observe that,
\begin{eqnarray*}
 \left(\left|\prod_{i=1}^k g_{\mu^{(i)}}(Y) X^\sigma\right|_\lambda\right) p^{c\cdot \sigma} &=& \prod_{i=1}^k \left(\left|g_{\mu^{(i)}}(Y) X^\sigma\right|_\lambda p^{c\cdot\mu^{(i)}}\right)
\end{eqnarray*}
 converges to $0$ as $|\sigma|$ approaches $\infty$ because $g_\mu = -c_0f_\mu$.  Suppose $\|f\|_{\lambda, c} \leq 1$, then this product is also less than or equal to one, because each term satisfies this property, so $\|f^{-1}\|_{\lambda, c} = |c_0^{-1}|_\lambda = 1$.

\end{proof}
\begin{prop}
The ring,  $(F[X; Y, \lambda]_c, \| \; \|_{\lambda, c})$, is an $F[[Y]]$-Banach algebra.
\end{prop}
\begin{proof}
  Suppose that $f = \sum_{\mu} f_\mu(Y)X^\mu$ and $g = \sum_{\nu} g_\nu(Y)X^\nu$, are elements of $F[X; Y, \lambda]_c$, then $|f_\mu \pm g_\mu|_\lambda \leq \max\{|f_\mu|_\lambda, |g_\mu|_\lambda\}$, and the quantity $\max\{|f_\mu|_\lambda p^{c\cdot \mu}, |g_\mu|_\lambda p^{c\cdot \mu}\}$ converges to 0 as $|\mu|$ approaches $\infty$.  Thus, $f+g$ is an element of $F[X; Y, \lambda]_c$.

 Similarly, we see that
\begin{eqnarray*}
 \left|\sum_{\mu + \nu = \sigma}f_\mu g_\nu\right|_\lambda \leq \max_{\mu + \nu = \sigma}\{|f_\mu|_\lambda\cdot |g_\nu|_\lambda\}
\end{eqnarray*}
and $\lim_{|\sigma| \rightarrow \infty} \max_{\mu + \nu = \sigma}\{|f_\mu|\cdot |g_\nu|p^{c\cdot \sigma}\} = 0$ as desired. Thus, $fg$ is an element of $F[X; Y, \lambda]_c$.

Now to prove that that this norm is complete, we let $\left(f^{(i)}\right)_{i=1}^\infty = \left(\sum_{\mu}^\infty f_\mu^{(i)} X^\mu\right)_{i=1}^\infty$ be a Cauchy sequence in $F[X; Y, \lambda]_c$.  Then we can choose a suitable subsequence of $\left(f^{(i)}\right)_{i=1}^\infty$ (because a Cauchy sequence is convergent if, and only if, it has a convergent subsequence) and assume that
\begin{eqnarray*}
 | f_\mu^{(j)} - f_\mu^{(i)} |_\lambda p^{c\cdot \mu} \leq \|f^{(j)} - f^{(i)}\|_{\lambda, c} < 1/i \quad \text{ for all } j > i > 0\text{ and all } |\mu|\geq 0.
\end{eqnarray*}
For all $j$ and $\mu$, $f_\mu^{(j)}$ is an element of $F[[Y]]$ which is complete, so there is an element $f_\mu$ in $F[[Y]]$ such that $f_\mu^{(j)}$ converges to  $f_\mu$ as $j$ approaches $\infty$. Define $f =\sum_{\mu} f_\mu X^\mu$.  We claim that $|f_\mu|_\lambda p^{c\cdot\mu}$ converges to $0$ as $|\mu|$ approaches $\infty$.

Note that $| \; |_\lambda$ is continuous, so $|f_\mu - f_\mu^{(i)}|_\lambda p^{c\cdot \mu} \leq 1/i$, for all $|\mu| \geq 0$ and all $i > 0$.  We choose $\mu$, such that $|\mu|$ is sufficiently large, so that $|f_\mu^{(i)}|_\lambda p^{c\cdot \mu} < 1/i$. Since the norm is non-Archimedean, this shows that $|f_\mu|_\lambda p^{c\cdot \mu} \leq 1/i$.  Thus, $|f_\mu|_\lambda p^{c\cdot \mu}$ converges to $0$ as $|\mu|$ approaches $\infty$. Hence, $f$ is an element of $F[X; Y, \lambda]_c$ and $\|f- f^{(i)}\|_{\lambda,c} = \max |f_\mu - f^{(i)}_\mu|_\lambda p^{c\cdot \mu} \leq  1/i$, therefore, $\lim_i f^{(i)} = f$.
\end{proof}
\begin{defi}
 A power series $f(X,Y) = \sum_{k=0}^\infty f_k(X,Y)X_n^k$ in $F[X; Y, \lambda]$ is called $X_n$-distinguished of degree $s$ in $F[X; Y, \lambda]$ if
\begin{enumerate}
 \item  $f_s(X,Y)$ is a unit in $F[X_1,\cdots, X_{n-1}; Y, \lambda]$ and
\item $|f_k(X, Y)| < 1$ for all $k > s$.
\end{enumerate}
Equivalently, $f \mod (Y)$ is a unitary polynomial in $X_n$ of degree $s$.

A power series $f(X, Y) = \sum_{k=0}^\infty f_k(X,Y)X_n^k$ in $F[X; Y, \lambda]_c$ is called $X_n$-distinguished of degree $s$ in $F[X; Y, \lambda]_c$ if
\begin{enumerate}
 \item $f_s(X,Y)$ is a unit in $F[X_1,\cdots, X_{n-1}; Y, \lambda]_c$ and
 $\|f_s(X,Y)\|_{\lambda, (c_1, \cdots, c_{n-1})} = 1$,
\item $\|f\|_{\lambda, c} = \|f_s(X,Y)X_n^{s}\|_{\lambda, c} =
p^{c_ns} > \|f_k(X,Y)X_n^k\|_{\lambda, c}$ for all $k \neq s$.
\end{enumerate}
\end{defi}

 If an element $f$ in $F[X; Y, \lambda]$ is $X_n$-distinguished of degree $s$ in $F[X; Y, \lambda]$, then it is $X_n$-distinguished in $F[X; Y, \lambda]_c$ for some $c$ in $\mr^n.$  Indeed, suppose that $f$ is an element of $F[X; Y,\lambda]_c$. Since $f_s(X,Y)$ is a unit, we can write $f_s(X,Y) = u + h$, where $u$ is a unit in $F[[Y]]$, and $h$ is an element of $(Y)$.  By choosing $c_1, \cdots, c_{n-1}$ small enough, we can make $\|h\|_{\lambda, c} < 1$, and so $\|f_s(X,Y)\|_{\lambda,  (c_1, \cdots, c_{n-1})} = 1$.  We can reduce $c_1, \cdots, c_{n-1}$ even further to ensure that $\|f_k(X,Y)X_n^{k}\|_{\lambda, c} < \|f_s(X,Y)X_n^s\|_{\lambda, c} = p^{c_ns}$, because $f_k$ is an element of $(Y)$ for all $k > s$.  Now, to ensure that $\|f_k(X,Y)X_n^{k}\|_{\lambda, c} < \|f_s(X,Y)X_n^s\|_{\lambda, c} = p^{c_ns}$, for $k < s$, we can shrink $c_1, \cdots, c_{n-1}$ once again so that $\|f_k(X,Y)\|_{\lambda, c} < p^{c_n}$.  In this way we find that for all $k < s$, $\|f_k(X,Y)X_n^k\|_{\lambda, c} < p^{c_n} p^{c_n(s-1)} =  p^{c_ns} = \|f_s(X,Y)X_n^{s}\|_{\lambda, c}$ as desired.

We can use the notion of $X_n$-distinguished elements to derive a Euclidean algorithm for $F[X; Y, \lambda]_c$.  This Euclidean algorithm will then produce Weierstrass factorization for $X_n$-distinguished elements, which will allow us to deduce that $F[X; Y, \lambda]$ is noetherian and factorial.
\begin{thm}
 Let
\begin{eqnarray*}
 g &=& \sum_{k=0}^\infty g_k(X,Y)X_n^k
\end{eqnarray*}
be $X_n$-distinguished of degree $s$ in $F[X; Y, \lambda]_c$. Then every $f$ in $F[X; Y, \lambda]_c$ can be written uniquely in the form
\begin{eqnarray*}
 f = qg + r
\end{eqnarray*}
where $q$ is and element of $F[X; Y, \lambda]_c$ and $r$ is a polynomial in $F[X_1, \cdots, X_{n-1}; Y, \lambda]_c[X_n]$, with $\deg_{X_n}(r) < s$. Further, if $f$ and $g$ are polynomials in $X_n$, then so are $q$ and $r$.
\end{thm}
\begin{proof}

Let $\alpha$, $\tau$ be projections given by,
\begin{eqnarray*}
 \alpha : \sum_{k=0}^\infty g_k(X, Y)X_n^k &\mapsto& \sum_{k=0}^{s-1} g_k(X,Y)X_n^k\\
\tau : \sum_{k=0}^\infty g_k(X,Y)X_n^k &\mapsto& \sum_{k=s}^\infty g_k(X,Y)X_n^{k-s} \\
\end{eqnarray*}
We see that $\tau(w)$ and $\alpha(w)$ are elements of $F[X;Y,\lambda]_c$, and that $\tau(wX_n^{s}) = w$. It is also clear that $\tau(w) = 0$ if, and only if, $\deg_{X_n}( w) < s$, for all $w$ in $F[X; Y, \lambda]_c$.

 Such $q$ and $r$ exist if, and only if, $\tau(f) = \tau(qg)$.  Thus, we must solve,
\begin{eqnarray*}
 \tau(f) = \tau(q\alpha(g)) + \tau(q\tau(g)X_n^{s}) = \tau(q\alpha(g)) + q\tau(g).
\end{eqnarray*}
Note that $\tau(g)$ is invertible, trivially, because it is congruent to a unit modulo $(Y)$.  Let $M = q\tau(g)$.  Thus, we can write
\begin{eqnarray*}
 \tau(f) = \tau\left(M\frac{\alpha(g)}{\tau(g)}\right) + M = \left(I + \tau\circ \frac{\alpha(g)}{\tau(g)}\right)M.
\end{eqnarray*}
We want to show that the map $\left(I + \tau\circ \frac{\alpha(g)}{\tau(g)}\right)^{-1}$ exists.

Suppose $z$ is an element of $F[X; Y, \lambda]_c.$  We first claim that $\|\tau(z)\|_{\lambda, c} \leq \frac{\|z\|_{\lambda, c}}{p^{c_ns}}$. Indeed, there exists $\mu$ such that
\begin{eqnarray*}
 \|z\|_{\lambda, c} &=& |z_\mu|_\lambda p^{c\cdot \mu} \\
&\geq& |z_\nu|_\lambda p^{c\cdot\nu} \quad \text{ for all } \nu
\end{eqnarray*}
Thus,
\begin{eqnarray*}
 \frac{\|z\|_{\lambda, c}}{p^{c_ns}} &=& |z_\mu|_\lambda p^{c\cdot \mu - c_ns} \\
&\geq& |z_\nu|_\lambda p^{c\cdot \nu - c_ns} \quad \text{ for all } \nu.
\end{eqnarray*}
The maximum over all $\nu$ with $\nu_n \geq s$ is equal to $\|\tau(z)\|_{\lambda, c}$, as asserted. Thus $\|\tau(g)\|_{\lambda, c} \leq \frac{p^{c_ns}}{p^{c_ns}} = 1$, so by the lemma 2.8 $\|\tau(g)^{-1}\| \leq 1$.  Let $h = \frac{\alpha(g)}{\tau(g)}$. Then,
\begin{eqnarray*}
 \|h\|_{\lambda, c} \leq \|\alpha(g)\|_{\lambda, c}\|\tau(g)^{-1}\|_{\lambda, c} < p^{c_ns}.
\end{eqnarray*}

Next we claim that $\|(\tau\circ h)^m (z)\|_{\lambda, c} < \frac{\|z\|_{\lambda, c}\|h\|_{\lambda, c}^{m}}{p^{mc_ns}}$, for all $m$ in $\mn$.  Indeed, $\|\tau(zh)\|_{\lambda, c} \leq \frac{\|zh\|_{\lambda, c}}{p^{c_ns}} \leq \frac{\|z\|_{\lambda, c}\|h\|_{\lambda, c}}{p^{c_ns}}$ by what we just proved.  Now, assume that this is true for $m$, then
\begin{eqnarray*}
 \|\tau\left((\tau\circ h)^{m}(z) h \right)\|_{\lambda, c} &\leq& \frac{\|(\tau\circ h)^m(z)\|_{\lambda, c}\|h\|_{\lambda, c}}{p^{c_ns}} \\
&\leq& \frac{\|z\|_{\lambda, c}\|h\|_{\lambda, c}^{m}}{p^{mc_ns}}\frac{\|h\|_{\lambda, c}}{p^{c_ns}} \\
&=& \frac{\|z\|_{\lambda, c}\|h\|_{\lambda, c}^{m+1}}{p^{(m+1)c_ns}}
\end{eqnarray*}
Now, we know that,
\begin{eqnarray*}
\left(I + \tau\circ h\right)^{-1}(z) &=& z + \sum_{m=1}^\infty (-1)^{m} (\tau\circ h)^m(z).
\end{eqnarray*}
Let $w^{(i)}(z) = z + \sum_{m=1}^i (-1)^{m} (\tau\circ h)^m(z)$. We claim that the sequence $\left(w^{(i)}(z)\right)_{i=1}^\infty$ is Cauchy for every $z$ in $F[X; Y, \lambda]_c$.  Indeed,
\begin{eqnarray*}
 \|w^{(i+1)}(z) - w^{(i)}(z)\|_{\lambda, c} &=& \|(\tau\circ h)^{i+1}(z)\|_{\lambda, c} \\
&\leq& \frac{\|z\|_{\lambda, c}\|h\|_{\lambda, c}^{i+1}}{p^{(i+1)c_ns}}.
\end{eqnarray*}
Since $\|h\|_{\lambda, c} < p^{c_ns}$, we see that this difference approaches $0$ as $i$ approaches $\infty$.  Since this norm is non-Archimedean, this is all we need to show.  Therefore, since $F[X; Y, \lambda]_c$ is complete, we see that $w(z) = \lim_i w^{(i)}(z)$ exists for every $z$ in $F[X; Y, \lambda]_c$.  Uniqueness is immediate from the invertibility of the map.

To prove the last statement, note that we could already carry out division uniquely in the ring $F[X_1, \cdots, X_{n-1}; Y,
\lambda]_c[X_n]$, by the polynomial Euclidean algorithm. Therefore, the division is unique in $F[X; Y, \lambda]_c$.
\end{proof}
\begin{cor}
 Suppose that $g$ is $X_n$-distinguished of degree $s$ in $F[X; Y, \lambda]$, and that $f$ is an element of $F[X; Y, \lambda]$.  Then there exist unique elements, $q$ in $F[X; Y, \lambda]$, and $r$ in the polynomial ring $F[X_1, \cdots, X_{n-1}; Y, \lambda][X_n]$ with ${\rm deg}_{X_n}(r) <s$, such that $f = qg + r$.
\end{cor}
\begin{proof}
 Choose $c$ for which $f$ and $g$ are elements of $F[X; Y, \lambda]_c$. Then, choose $c_i' \leq c_i$ such that $g$ is $X_n$-distinguished of degree $s$ in $F[X; Y, \lambda]_{c'}$.  Carry out the division in this ring.

To show uniqueness we observe that $|f| = \max\{|q|, |r|\}$.  Indeed, without loss of generality we can assume that $\max\{|q|, |r|\} = 1$. Thus, $|f| \leq 1$.  Suppose that $|f| < 1$. Then, $0 \equiv qg + r \mod (Y)$.  Since $\deg_{X_n} (r) < s = \deg_{X_n} (g \mod (Y))$, we must have that $q = r \equiv 0 \mod (Y)$, contradicting $\max\{|q|, |r|\} = 1$.  Thus, if $q'g + r' = qg + r$, then $(q-q')g + (r-r') = 0$, thus, $|q-q'| = |r-r'| = 0$.
\end{proof}

\begin{thm}
 Let $f$ be $X_n$-distinguished of degree $s$. Then, there exists a unique monic polynomial $\omega$ in $F[X_1, \cdots, X_{n-1}; Y, \lambda][X_n]$ of degree $s$ in $X_n$ and a unique unit $e$ in $F[X ; Y, \lambda]$ such that $f = e\cdot \omega $.  Further, $\omega$ is distinguished of degree $s$.
\end{thm}
\begin{proof}
 By the previous theorem there exists an element $q$ in $F[X; Y, \lambda]$ and a polynomial $r$ in $F[X_1, \cdots, X_{n-1}; Y, \lambda][X_n]$ such that $\deg_{X_n}( r) < s$ and $X_n^s = qf + r$.  We let $\omega = X_n^s - r$.  $\omega = qf$ is clearly $X_n$-distinguished of degree $s$.  Since the $X_n$-degree of $\omega \mod (Y)$ is the same the $X_n$-degree of $f \mod (Y)$, we see that $q$ is a unit.  Set $e = q^{-1}$, yielding $f = e\cdot \omega$.  Uniqueness is immediate from the uniqueness of the division algorithm.
\end{proof}
\begin{defi}
 A Weierstrass polynomial is a monic $X_n$-distinguished polynomial in $F[X_1, \cdots, X_{n-1}; Y, \lambda][X_n]$.
\end{defi}
\begin{lem}
 Let $\sigma$ be the map such that $\sigma(X_i) = (X_i + X_j^d)$ with $d\geq 1$, and $\sigma(X_j) = X_j$ for all $j \neq i$.  Then $\sigma$ is a well-defined automorphism of $F[X; Y, \lambda]$.
\end{lem}
\begin{proof}
 Without loss of generality, we can assume that $i = n$ and $j = 1$. Let $f(X,Y) = \sum_{\mu} f_\mu(Y)X^\mu = \sum_{\mu} f_\mu(Y)X_1^{\mu_1}\cdots X_n^{\mu_n}$.  Observe that,
\begin{eqnarray*}
 \sigma(f) &=& \sum_\mu f_\mu(Y)X_1^{\mu_1}\cdots X_{n-1}^{\mu_{n-1}}(X_n + X_1^d)^{\mu_n} \\
&=& \sum_\mu f_\mu(Y)X_1^{\mu_1}\cdots X_{n-1}^{\mu_{n-1}}\sum_{j=0}^{\mu_n} \binom{\mu_{n}}{j}X_1^{dj}X_n^{\mu_n-j}.
\end{eqnarray*}
The quantity $dj + \mu_n - j= (d-1)j +\mu_n$ is maximal when $j = \mu_n$. Therefore, the above converges if
\begin{eqnarray*}
 \sum_\mu f_\mu(Y)X_1^{\mu_1 + d\mu_n}\cdots X_{n-1}^{\mu_{n-1}}X_n^{\mu_n}
\end{eqnarray*}
converges.  Choose $c$ so that $|f_\mu(Y)|p^{c\cdot \mu}$ converges to $0$ as $|\mu|$ approaches $\infty$.  Let $c_n' = \frac{c_n}{2(d+1)}$. Then, $c_1\mu_1 + \cdots + c_n\mu_n > c_1\mu_1 + \cdots + c_n'\mu_{n}(d+1)$, so $|f_\mu(Y)|p^{c_1\mu_1 + \cdots + c_n'\mu_{n}(d+1)}$ converges to $0$ as $|\mu|$ approaches $\infty$. Therefore, this map is well defined with inverse, $\sigma^{-1}(X_n) = X_n - X_1^d$ and $\sigma^{-1}(X_j) = X_j$, if $j \neq n$.
\end{proof}
\begin{thm}
 Suppose $f(X,Y) = \sum_{\mu}f_\mu(Y) X^\mu$ is an element of $F[X; Y,\lambda]$.  If $|f_\mu(Y)| = 1$, for some $\mu$, where $\mu_n > 0$, then there exists an automorphism $\sigma$ of $F[X; Y, \lambda]$ such that $\sigma(f)$ is $X_n$-distinguished.
\end{thm}
\begin{proof}
 Let $f(X,Y) = \sum_{\mu} f_\mu(Y) X^\mu = \sum_{\mu} f_\mu(Y)X_1^{\mu_1}\cdots X_n^{\mu_n}$.  Let $\nu = (\nu_1, \cdots, \nu_n)$ be the maximal $n$-tuple, with respect to lexicographical ordering, such that $f_\nu(Y)$ is not an element of $(Y)$.  Let $t \geq \max_{1\leq i\leq n} \mu_i$ for all indices $\mu$ such that $f_\mu(Y)$ is not an element of $(Y)$, e.g., let $t$ be the total $X$-degree of $f(X,Y) \mod (Y)$. Now, define an automorphism $\sigma(X_i) = X_i + X_n^{d_i}$ for $i = 1, \cdots, n-1$, and $\sigma(X_n) = X_n$, where $d_n = 1$, and $d_{n-j} = 1 + t\sum_{k=0}^{j-1} d_{n-k}$, for $j = 1, \cdots, n-1$.  We see that this map is just a finite composition of automorphisms of the same type as given above. Hence, it is an automorphism.

We will prove that $\sigma(f)$ is $X_n$-distinguished of degree $s = \sum_{i=1}^nd_i\nu_i$.  First, for all $\mu$ such that $f_\mu(Y)$ is a unit, and $\mu \neq \nu$, we have $\sum_{i=1}^n d_i\mu_i < s$:  There exists an index $q$ such that $1 \leq q \leq n$, such that $\mu_1 = \nu_1, \cdots, \mu_{q-1} = \nu_{q-1}$ and $\mu_q < \nu_q$.  Therefore $\mu_q \leq \nu_q - 1$ and
\begin{eqnarray*}
 \sum_{i=1}^nd_i\mu_i \leq \sum_{i=1}^{q-1}d_i\nu_i + d_q(\nu_q-1) + t\sum_{i=q+1}^nd_i =
 \sum_{i=1}^qd_i\nu_i - 1 < \sum_{i=1}^nd_i\nu_i = s.
\end{eqnarray*}
Now,
\begin{eqnarray*}
 \sigma(f) &=& \sum_\mu f_\mu(Y)(X_1 + X_n^{d_1})^{\mu_1}\cdots(X_{n-1} + X_n^{d_{n-1}})^{\mu_{n-1}}X_n^{\mu_n} \\
&\equiv& \sum_{\substack{\mu\\  f_{\mu}(Y) \notin (Y)}}f_{\mu}(Y)\sum_{\substack{\lambda_1, \cdots, \lambda_{n-1} \\ 0\leq \lambda_i \leq \mu_i}}\binom{\mu_1}{\lambda_1}\cdots\binom{\mu_{n-1}}{\lambda_{n-1}}X_1^{\mu_1 - \lambda_1}\cdots X_{n-1}^{\mu_{n-1} - \lambda_{n-1}}X_n^{d_1\lambda_1 + \cdots + d_{n-1}\lambda_{n-1} + \mu_n} \\
&\equiv& \sum g_iX^i_n \mod(Y)
\end{eqnarray*}
where the $g_i$ are elements of $F[X_1, \cdots, X_{n-1}]$.  Therefore, $\sigma(f) \mod (Y)$ is a polynomial in $X_n$ of degree less than or equal to $s$, and $X_n^{d_1\lambda_1 + \cdots + d_{n-1}\lambda_{n-1} + \mu_n} = X_n^s$ if, and only if, $\mu_n = \nu_n$ and $\lambda_i = \mu_i = \nu_i$ for $i = 1, \cdots, n-1$.  Thus, we have $g_s = f_\nu(Y) \mod(Y)$, but $f_\nu(Y)$ is not an element of $(Y)$, and so $\sigma(f)$ is a unitary polynomial modulo $(Y)$.  Therefore, $\sigma(f)$ is $X_n$-distinguished of degree $s$.
\end{proof}
\begin{thm}
 Let $\omega$ be a Weierstrass polynomial of degree $s$ in $X_n$.  Then for all $d \geq 0$
\begin{enumerate}
 \item $Y^dF[X; Y, \lambda]/Y^d\omega F[X; Y, \lambda]$ is a finite free $F[X_1, \cdots, X_{n-1}; Y, \lambda]$-module, and
\item $Y^dF[X_1, \cdots, X_{n-1}; Y, \lambda][X_n]/Y^d\omega F[X_1, \cdots, X_{n-1}; Y, \lambda][X_n]$ is isomorphic to \\$Y^dF[X; Y, \lambda]/Y^d\omega F[X; Y, \lambda]$.
\end{enumerate}
\end{thm}
\begin{proof}
 Suppose that $g$ is an element of $Y^dF[X; Y, \lambda]$, then $g = Y^dh$ for some element $h$ in $F[X; Y, \lambda]$. Since $\omega$ is $X_n$-distinguished, there exists a unique element $q$ in $F[X; Y, \lambda]$, and a unique polynomial $r$ in $F[X_1, \cdots, X_{n-1}; Y, \lambda][X_n]$ with $\deg_{X_n} (r) < s$, such that $h = q\omega + r$, so $g = qY^d\omega + Y^dr$.  Therefore, $g \equiv Y^dr \mod Y^d\omega F[X; Y, \lambda]$, so the set $\{Y^d, Y^dX_n, \cdots, Y^dX_n^{s-1}\}$ forms a generating set of $Y^dF[X; Y, \lambda]/Y^d\omega F[X; Y, \lambda]$ over the ring $F[X_1,\cdots, X_{n-1}; Y, \lambda]$. The natural map 
$$Y^dF[X_1, \cdots, X_{n-1}; Y, \lambda][X_n] \rightarrow Y^dF[X; Y, \lambda]/Y^d\omega F[X; Y, \lambda]$$
 is thus surjective.  The kernel of this map is $Y^d \omega F[X_1, \cdots, X_{n-1}; Y, \lambda][X_n]$, trivially.
\end{proof}
\begin{thm}
 $F[X; Y, \lambda] = F[X_1, \cdots, X_n; Y, \lambda]$ is factorial, for all $n \geq 1$.
\end{thm}
\begin{proof}
 First assume that $n=1$.  Suppose that $f$ is an element of $F[X; Y, \lambda]$.  Write $f = e\cdot Y^d  \omega $, where $\omega$ is a  unitary polynomial in $X$ of degree $s$ in $F[[Y]][X]$, and $e$ is a unit in $F[X; Y, \lambda]$.  We can factor $\omega = uq_1\cdots q_m$ into irreducible factors and a unit in $F[[Y]][X]$ because this ring is factorial. We want to show that these factors are still irreducible in $F[X; Y, \lambda]$.  Suppose that $q_i$ is not irreducible modulo $\omega F[[Y]][X]$, then $q_i \equiv ab \mod \omega$, so there exists $g \neq 0$ such that $q_i = ab + g\omega$.  However, by the uniqueness of the division algorithm $g = 0$, thus, $a$ or $b$ is a unit modulo $\omega$. Therefore, $q_i$ is irreducible in $F[[Y]][X]/\omega F[[Y]][X] \simeq F[X; Y, \lambda]/\omega F[X; Y, \lambda]$.

If $q_i$ is not irreducible in $F[X; Y, \lambda]$, then there exists elements $a$ and $b$ in $F[X; Y, \lambda]$, such that $q_i = ab$. Without loss of generality, $b$ must be a unit modulo $\omega$, so $b = c_0 + g\omega$.  Write $q_i = a(c_0 + g\omega) = ac_0 + ag\omega$.  However, by the uniqueness of the division algorithm, the same representation of the division algorithm which holds in $F[[Y]][X]$, holds in $F[X; Y, \lambda]$, and since $\deg_X (ac_0) < s$, we must have $ag = 0$. This is a contradiction.  Therefore, the $q_i$ are irreducible in both rings.  Write $f = eu\cdot Y^dq_1\cdots q_m$ uniquely as a product of irreducible factors and a unit. Continue by induction.
\end{proof}
\begin{thm}
 $F[X_1, \cdots, X_n; Y, \lambda]$ is noetherian.
\end{thm}
\begin{proof} Assume first that $n =1$.
Let $I \subseteq F[X; Y, \lambda]$ be an ideal.  Suppose that $d$ is the largest positive integer such that $I \subseteq Y^d F[X; Y,\lambda]$. Then every $f$ in $I$ is divisible by $Y^d$. Choose an element $f$ in $I$ such that $\ord_Y f = d$. We can then write $f = e\cdot Y^d \omega $ for some unit $e$, and Weierstrass polynomial $\omega$. Consider the image of $I$ in $Y^dF[X; Y, \lambda]/Y^d\omega F[X; Y, \lambda] \simeq Y^dF[[Y]][X]/Y^d\omega F[[Y]][X]$; this is Noetherian. Therefore, we can pull back the finite list of generators for the image of $I$ and add $Y^d \omega$ to get a finite generating system for $I$.  Continue by induction.

\end{proof}
\section{Further Questions}

This paper resolves the open problem left in Wan
\cite{Wan1}, stated at the beginning of the paper, only when $F$ is a field and when $F[X; Y, \lambda]$ has only one $Y$ variable. It would be interesting to settle the general case (either positively or negatively) when $Y$ has more than one variable and $R$ is a general noetherian ring.

Another open question is whether $F[X; Y, \lambda]$ is factorial if there is more than one $Y$ variable. The answer to this
question cannot be obtained from the same methods used in this paper because elements exist that cannot be transformed into an $X_n$ distinguished element through an automorphism. For example:
\begin{eqnarray*}
 f(X,Y) = Y_1 + XY_2 + X^2Y_1^2 + X^3Y_2^3 + \cdots .
\end{eqnarray*}

 Another direction of research could involve studying the algebras $T_n(\rho, \lambda)$ and $W_n(\lambda)$. One could try to
generalize results only known about the overconvergent case $(\lambda(x) = id)$, such as those proven in Gross-Kl\"{o}nne \cite{Gr}. One could also try to develop the $k$-affinoid theory of $T_n(\rho, \lambda)$ and $W_n(\lambda)$.

\end{document}